\documentclass[letterpaper, 10 pt, conference]{ieeeconf}  %

\IEEEoverridecommandlockouts                              %

\usepackage{textcomp,hyperref}
\newcommand\copyrighttext{%
	\footnotesize \textcopyright 2025 IEEE. Personal use of this material is permitted.
	Permission from IEEE must be obtained for all other uses, in any current or future 
	media, including reprinting/republishing this material for advertising or promotional 
	purposes, creating new collective works, for resale or redistribution to servers or 
	lists, or reuse of any copyrighted component of this work in other works. 
}
\newcommand\copyrightnotice{%
	\begin{tikzpicture}[remember picture,overlay]
		\node[anchor=south,yshift=10pt] at (current page.south) {\fbox{\parbox{\dimexpr\textwidth-\fboxsep-\fboxrule\relax}{\copyrighttext}}};
	\end{tikzpicture}%
}

\usepackage{graphics} %
\usepackage{epsfig} %
\usepackage{mathptmx} %
\usepackage{times} %
\usepackage{amsmath} %
\usepackage{amssymb}  %

\usepackage[ansinew]{inputenc} %
\usepackage[T1]{fontenc} %
\usepackage{xcolor}
\usepackage{MnSymbol}
\usepackage{mathtools}
\usepackage{bm}
\usepackage{booktabs}
\usepackage{microtype} %
\usepackage{siunitx}
\usepackage{tabularx, tabulary}
\usepackage{ragged2e}
\usepackage{multirow}
\usepackage{algorithm,algpseudocode}
\let\labelindent\relax %
\usepackage{enumitem}
\usepackage{hyperref}
\usepackage{cite}
\usepackage{standalone}

\usepackage{tikz}
\newcommand\tikznode[3][]%
{\tikz[remember picture,baseline=(#2.base)]
	\node[minimum size=0pt,inner sep=0pt,#1](#2){#3};%
}
\usepackage{calc}
\usetikzlibrary{calc} 

\usetikzlibrary{decorations} %
\usetikzlibrary{matrix} %
\usetikzlibrary{arrows} %
\usetikzlibrary{backgrounds, fit}
\usepackage{tikz-layers}

\usetikzlibrary{positioning,shapes.multipart,shapes.arrows, decorations.markings,arrows.meta,shapes.symbols}

\tikzset{
	-|-/.style={
		to path={
			(\tikztostart) -| ($(\tikztostart)!#1!(\tikztotarget)$) |- (\tikztotarget)
			\tikztonodes
		}
	},
	-|-/.default=0.5,
	|-|/.style={
		to path={
			(\tikztostart) |- ($(\tikztostart)!#1!(\tikztotarget)$) -| (\tikztotarget)
			\tikztonodes
		}
	},
	|-|/.default=0.5,
}

\tikzstyle{block} = [draw=black,fill=white,rectangle,thick,minimum height=2em,minimum width=4em, align=center]
\tikzstyle{sum} = [draw,circle,inner sep=0mm,minimum size=2mm]
\tikzstyle{connector} = [->,thick]
\tikzstyle{dashedconnector} = [->,thick,dashed]
\tikzstyle{line} = [thick]
\tikzstyle{dashedline} = [thick,dashed]
\tikzstyle{branch} = [circle,inner sep=0pt,minimum size=1mm,fill=black,draw=black]
\tikzstyle{guide} = []
\tikzstyle{snakeline} = [connector, decorate, decoration={pre length=0.2cm,
	post length=0.2cm, snake, amplitude=.4mm,
	segment length=2mm},thick, magenta, ->]

\newcommand{\gettikzxy}[3]{%
	\tikz@scan@one@point\pgfutil@firstofone#1\relax
	\edef#2{\the\pgf@x}%
	\edef#3{\the\pgf@y}%
}

\makeatletter
\pgfkeys{/pgf/.cd,
	parallelepiped offset x/.initial=2mm,
	parallelepiped offset y/.initial=2mm
}
\pgfdeclareshape{parallelepiped}
{
	\inheritsavedanchors[from=rectangle] %
	\inheritanchorborder[from=rectangle]
	\inheritanchor[from=rectangle]{north}
	\inheritanchor[from=rectangle]{north west}
	\inheritanchor[from=rectangle]{north east}
	\inheritanchor[from=rectangle]{center}
	\inheritanchor[from=rectangle]{west}
	\inheritanchor[from=rectangle]{east}
	\inheritanchor[from=rectangle]{mid}
	\inheritanchor[from=rectangle]{mid west}
	\inheritanchor[from=rectangle]{mid east}
	\inheritanchor[from=rectangle]{base}
	\inheritanchor[from=rectangle]{base west}
	\inheritanchor[from=rectangle]{base east}
	\inheritanchor[from=rectangle]{south}
	\inheritanchor[from=rectangle]{south west}
	\inheritanchor[from=rectangle]{south east}
	\backgroundpath{
		\southwest \pgf@xa=\pgf@x \pgf@ya=\pgf@y
		\northeast \pgf@xb=\pgf@x \pgf@yb=\pgf@y
		\pgfmathsetlength\pgfutil@tempdima{\pgfkeysvalueof{/pgf/parallelepiped
				offset x}}
		\pgfmathsetlength\pgfutil@tempdimb{\pgfkeysvalueof{/pgf/parallelepiped
				offset y}}
		\def\ppd@offset{\pgfpoint{\pgfutil@tempdima}{\pgfutil@tempdimb}}
		\pgfpathmoveto{\pgfqpoint{\pgf@xa}{\pgf@ya}}
		\pgfpathlineto{\pgfqpoint{\pgf@xb}{\pgf@ya}}
		\pgfpathlineto{\pgfqpoint{\pgf@xb}{\pgf@yb}}
		\pgfpathlineto{\pgfqpoint{\pgf@xa}{\pgf@yb}}
		\pgfpathclose
		\pgfpathmoveto{\pgfqpoint{\pgf@xb}{\pgf@ya}}
		\pgfpathlineto{\pgfpointadd{\pgfpoint{\pgf@xb}{\pgf@ya}}{\ppd@offset}}
		\pgfpathlineto{\pgfpointadd{\pgfpoint{\pgf@xb}{\pgf@yb}}{\ppd@offset}}
		\pgfpathlineto{\pgfpointadd{\pgfpoint{\pgf@xa}{\pgf@yb}}{\ppd@offset}}
		\pgfpathlineto{\pgfqpoint{\pgf@xa}{\pgf@yb}}
		\pgfpathmoveto{\pgfqpoint{\pgf@xb}{\pgf@yb}}
		\pgfpathlineto{\pgfpointadd{\pgfpoint{\pgf@xb}{\pgf@yb}}{\ppd@offset}}
	}
}
\makeatother

\tikzstyle{server}=[
parallelepiped,
fill=white, draw=white,
minimum width=0.35cm,
minimum height=0.75cm,
parallelepiped offset x=3mm,
parallelepiped offset y=2mm,
xscale=-1,
path picture={
	\draw[top color=white,bottom color=white]
	(path picture bounding box.south west) rectangle 
	(path picture bounding box.north east);
	\coordinate (A-center) at ($(path picture bounding box.center)!0!(path
	picture bounding box.south)$);
	\coordinate (A-west) at ([xshift=-0.575cm]path picture bounding box.west);
	\draw[ports]([yshift=0.1cm]$(A-west)!0!(A-center)$)
	rectangle +(0.2,0.065);
	\draw[ports]([yshift=0.01cm]$(A-west)!0.085!(A-center)$)
	rectangle +(0.15,0.05);
	\fill[white]([yshift=-0.35cm]$(A-west)!-0.1!(A-center)$)
	rectangle +(0.235,0.0175);
	\fill[white]([yshift=-0.385cm]$(A-west)!-0.1!(A-center)$)
	rectangle +(0.235,0.0175);
	\fill[white]([yshift=-0.42cm]$(A-west)!-0.1!(A-center)$)
	rectangle +(0.235,0.0175);
}
]

\tikzstyle{ports}=[
line width=0.3pt,
top color=white,
bottom color=white
]

\usepackage{pgfplots}
\pgfplotsset{compat=newest}
\pgfplotsset{plot coordinates/math parser=false}
\newlength\figureheight
\newlength\figurewidth 

\usepackage{amsmath,amssymb,amsfonts}    %
\usepackage{mathtools}
\usepackage{cuted}
\newcommand{\ubar}[1]{\text{\b{$#1$}}}

\newcommand{\R}{\mathbb{R}} 

\newcommand{\N}{\mathbb{N}}

\newcommand{\Cc}{\mathcal{C}}

\newcommand{\Xc}{\mathcal{X}}

\newcommand{\norm}[1]{\left\lVert#1\right\rVert}

\newcommand{\ceil}[1]{\left\lceil#1\right\rceil}

\newcommand{\innerproduct}[2]{\left\langle #1, #2 \right\rangle}

\definecolor{istblue}{rgb/cmyk}{0.0,0.25490,0.56862745/1,0.55,0.0,0.43}%
\definecolor{istgreen}{rgb/cmyk}{0.388235,0.83137,0.4431372/0.53,0,0.46,0.16}%
\definecolor{istorange}{rgb/cmyk}{1.0,0.5333,0.0667/0.0,0.46,0.93,0.0}%
\definecolor{istred}{rgb/cmyk}{0.9961,0.2902,0.2863/0.0,0.7,0.71,0.0}%
\definecolor{istlightblue}{rgb/cmyk}{0.3765, 0.6863, 1.0/0.62,0.31,0.0}%
\definecolor{istdarkblue}{rgb/cmyk}{0.1176,0.1804,0.8706/0.86,0.79,0.0,0.12}%
\definecolor{istdarkgreen}{rgb/cmyk}{0.0627,0.5882,0.2824/0.89,0.0,0.52,0.41}%
\definecolor{istdarkred}{rgb/cmyk}{0.529,0.031,0.075/0,0.94,0.86,0.47}%
\definecolor{istlogoblue}{rgb/cmyk/gray}{0,0,0.804/1,1,0,0.2/0}%
\definecolor{unianthrazit}{rgb/cmyk}{0.24314,0.26667,0.29804/0.5,0.35,0.25,0.70}%
\definecolor{unimiddleblue}{rgb/cmyk}{0,0.31765,0.61961/1,0.7,0,0}%
\definecolor{unilightblue}{rgb/cmyk}{0,0.74510,1/0.7,0,0,0}%

\newcommand{\diagdots}[3][-25]{%
	\rotatebox{#1}{\makebox[0pt]{\makebox[#2]{\xleaders\hbox{$\cdot$\hskip#3}\hfill\kern0pt}}}%
}

\newtheorem{assumption}{Assumption}

\newtheorem{theorem}{Theorem}
\newtheorem{proposition}{Proposition}
\newtheorem{lemma}{Lemma}
\newtheorem{remark}{Remark}

\newtheorem{requirement}{Requirement}

\newcommand{\commentout}[1]{}

\title{\LARGE \bf
A polynomial-based QCQP solver for encrypted optimization
}

\author{Sebastian Schlor, Andrea Iannelli, Junsoo Kim, Hyungbo Shim and Frank Allgöwer%
\thanks{F.\ Allgöwer, A.\ Iannelli, and S.\ Schlor are thankful that this work was funded by the Deutsche Forschungsgemeinschaft (DFG, German Research Foundation) under Germany's Excellence Strategy -- EXC 2075 -- 390740016. F.\ Allgöwer further thanks for the grants AL 316/13-2 -- 285825138 and AL 316/15-1 -- 468094890. 
Hyungbo Shim and Junsoo Kim gratefully acknowledge funding by the National Research Foundation of Korea (NRF) grant funded by the Korean Government (MSIT) under Grant RS-2022-00165417 and Grant RS-2024-00353032.
Parts of the paper were prepared when S.\ Schlor was staying at ASRI, Seoul National University, South Korea.
}%
\thanks{ S.\ Schlor, A.\ Iannelli, and F.\ Allgöwer are with the University of Stuttgart, Institute for Systems Theory and Automatic Control, Germany.  {\tt\small \{schlor, iannelli, allgower\}@ist.uni-stuttgart.de}.}%
\thanks{Junsoo Kim is with the Seoul National University of Science and Technology, Department of Electrical and Information Engineering, South Korea. {\tt\small junsookim@seoultech.ac.kr}.}%
\thanks{Hyungbo Shim is with the Seoul National University, ASRI, Department of Electrical and Computer Engineering, South Korea. {\tt\small hshim@snu.ac.kr}.}%
}

\begin{document}

\maketitle
\copyrightnotice
\thispagestyle{empty}
\pagestyle{empty}

\begin{abstract}
In this paper, we present a novel method for solving a class of quadratically constrained quadratic optimization problems using only additions and multiplications. This approach enables solving constrained optimization problems on private data since the operations involved are compatible with the capabilities of homomorphic encryption schemes.
To solve the constrained optimization problem, a sequence of polynomial penalty functions of increasing degree is introduced, which are sufficiently steep at the boundary of the feasible set. Adding the penalty function to the original cost function creates a sequence of unconstrained optimization problems whose minimizer always lies in the admissible set and converges to the minimizer of the constrained problem.
A gradient descent method is used to generate a sequence of iterates associated with these problems.
For the algorithm, it is shown that the iterate converges to a minimizer of the original problem, and the feasible set is positively invariant under the iteration.
Finally, the method is demonstrated on an illustrative cryptographic problem, finding the smaller value of two numbers, and the encrypted implementability is discussed. 

\end{abstract}

\section{Introduction}

Optimization algorithms are key elements of many modern technologies such as artificial intelligence, machine learning or model predictive control (MPC).
At the same time, these applications often deal with sensitive data and handle safety-critical tasks.
To enable privacy-preserving control and optimization, homomorphic cryptosystems have been adopted to perform computations entirely on encrypted data.
However, one of the limitations of this technology is that only polynomial operations, i.e., addition and multiplication are supported.
More complex operations, e.g., divisions and comparisons, have to be approximated, which leads to a higher computational effort and less accuracy. 
This challenge has prevented powerful state-of-the-art optimization algorithms to be applied
under encryption, 
in particular for constrained optimization problems, since these solvers generally need projections or comparisons, 
which cannot be natively handled in homomorphic cryptosystems.

In this paper, we present a novel approach to constrained optimization using only polynomial operations. %
This enables a straightforward implementation in a homomorphically encrypted fashion.
In the proposed algorithm, the constraints are replaced by a sequence of polynomial penalty functions, which are added to the original cost function.
To gradually approach the minimum of the original problem, a sequential gradient descent on the resulting sequence of unconstrained optimization problems is performed.

\subsection{Related work}

For convex unconstrained quadratic optimization problems, encrypted gradient and accelerated gradient methods have been analyzed and implemented in~\cite{Bertolace2023}, and distributed encrypted alternating direction method of multipliers (ADMM) has been proposed in~\cite{Binfet2023}.
The works~\cite{Shoukry2016,Alexandru2017,Alexandru2020, Alexandru2018} consider a linearly constrained quadratic program with private inputs from multiple parties.
The solution is obtained via a projected gradient ascent or projected fast gradient descent algorithm, where the critical projection operation is done by a target node in plaintext or a two-party protocol~\cite{Damgard2008}, involving more communication steps.
To solve quadratically constrained quadratic programs (QCQPs) occurring in encrypted MPC, there are approaches using real-time iterations of the proximal gradient method~\cite{Darup2018c} or ADMM~\cite{Darup2020c, Darup2019c,Darup2019d}.
Thereby, the projections are done by the plant in plaintext.
Alternatives to online optimization for MPC were proposed in~\cite{Darup2018,Darup2019d,Schluter2020,Tjell2021}, where an explicit MPC solution was computed offline and the identification of the active region was performed by the plant, a two-party protocol, or a garbled circuit.
All existing approaches to deal with constraints need a trusted party and decryption or two-party concepts, which are demanding from a computation or communication perspective.

For our approach, we utilize ideas from penalty and barrier methods. For an introduction, see, e.g.,~\cite{Nesterov2004}. Both types replace the constrained optimization problem by a sequence of unconstrained problems, which are easier to solve and approximate the solution of the original problem.
Penalty functions are usually defined as being zero inside the feasible set, and larger than zero outside~\cite{Fiacco1990}.
Barrier functions approach infinity at the boundary of the feasible set~\cite{Forsgren2002}.
The auxiliary problems then weigh the original cost function with a growing/decreasing influence of the penalty/barrier function.
Because penalty functions are often defined piece-wise, and the barrier functions have to grow unbounded on a finite domain, they cannot directly be used with polynomial-based homomorphic encryption.

\subsection{Contribution}
We propose a novel method to solve a class of constrained optimization problems only using polynomial operations (additions and multiplications). 
A detailed analysis of the algorithm is provided, in which convergence to a minimizer and positive invariance with respect to the constraints are shown.
In particular, we make the following contributions:
We introduce a sequence of polynomial penalty functions for convex quadratic constraint sets. When added to the original cost function, we obtain a sequence of unconstrained optimization problems.
For the resulting sequence of unconstrained problems, we prove that
\begin{itemize}
		\item each minimizer always lies inside the feasible set, and
		\item the sequence of minimizers converges to a minimizer of the original problem as the polynomial degree increases.
	\end{itemize}
To solve the original problem, we then propose a sequential gradient descent method generating a sequence of iterates associated with the aforementioned sequence of unconstrained problems.
For this algorithm, we show that
\begin{itemize}
	\item the feasible set is positively invariant, and %
	\item the iterates converge to a minimizer of the original problem.
\end{itemize}
Finally, we discuss its benefits for homomorphically encrypted implementations and give an illustrative example.

\subsection{Notation}
We define the natural numbers as $\N=\{1,2,3\dots\}$.
By $\innerproduct{x}{y}$ we denote the inner product of the vectors $x$ and $y$.
By $\succ 0$ ($\succeq 0$) we denote positive {(semi\nobreakdash-)} definiteness.
For positive semi-definite matrices, we use the Loewner order, i.e., we say that $A \succeq B$ if $A - B\succeq0$.
The boundary of a set $\Cc$ is denoted by $\partial\Cc$.
By $\sigma(A)$ we denote the set of singular values, and by $\bar{\sigma}(A)$ ($\ubar{\sigma}(A)$) the largest (smallest) singular value of $A$.
The smallest integer greater than or equal to a given number is obtained using the ceiling operation, denoted by $\ceil{\cdot}$.

\section{Problem setup and main idea}
In this section, we introduce the original constrained optimization problem and describe the overall approach, that is detailed in the following sections.

\subsection{Problem setup}

We want to solve the constrained optimization problem
\begin{equation}\label{eq:OriginalProblem}
	\begin{split}
	\Xc^\star =  \arg\min_{x} & \quad\hspace*{-0.5em} f(x)\\
	\text{s.t.} & \quad\hspace*{-0.5em} x\in \Cc
	\end{split}
\end{equation}
with the quadratic cost function
\begin{equation*}%
	f(x) = \frac{1}{2} x^\top Q x + q^\top x
\end{equation*}
and the constraint set
\begin{gather*}
	\begin{aligned}
		\Cc &= \{\,x\mid g(x)\leq 1\,\}, & \partial\Cc &= \{\,x\mid g(x)= 1\,\}, \label{eq:Cc}
	\end{aligned}\\
	g(x) = (x-v)^\top A (x-v) \label{eq:g}
\end{gather*}
with $x\in\R^n$, $Q\succeq 0$, $A\succ 0$ and $q,v\in\R^n$.
This is a special case of a convex QCQP.
Further, $\Cc$ should have nonzero volume, i.e., $A$ should be bounded.
Since $f$ is continuous and $\Cc$ is compact, the set of minimizers $\Xc^\star$ is nonempty and compact, but the minimizer might not necessarily be a singleton.

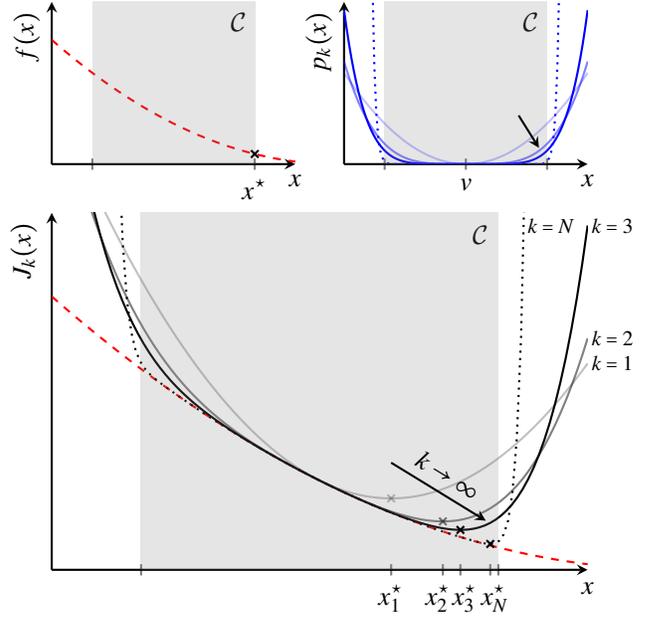
\begin{figure}[tb]
	\centering
	\setlength\figurewidth{0.8\columnwidth}
	\setlength\figureheight{0.5\columnwidth}
	\begin{tikzpicture}
	\setlength\figurewidth{0.75\columnwidth}
	\setlength\figureheight{0.5\columnwidth}
	\begin{axis}[name=plotF,
		anchor=north,
		scale only axis=true,
		height=0.5\figureheight,width=0.5\figurewidth,
		axis lines = left,
		axis line style = thick,
		x label style={at={(axis description cs:1,-0.01)},anchor=north},
		y label style={at={(axis description cs:-0.01,1)},anchor=south east},
		xlabel = \(x\),
		ylabel = {\(f(x)\)},
		ytick={},
		yticklabels={},
		ymajorticks=false,
		every major tick/.append style={thick},
xtick={-1,1},
xticklabels={{},{$x^\star$}},
		ymax=4, ymin=0,
		xmin=-1.5, xmax=1.5,
		every axis plot/.append style={thick}
		]
		\node[] at (axis cs: 0.8,3.5) {$\Cc$};
		\addplot [
		domain=-1.5:1.5, 
		samples=100, 
		color=red,
		dashed,
		]
		{0.25*x^2 - 1*x + 1};
		\addplot[mark=x, opacity=1] coordinates
		{(1,0.25)}; 
		\addplot[area legend, draw=none, fill=gray, fill opacity=0.2, forget plot]
		table[row sep=crcr] {%
			x	y\\
			-1	\pgfkeysvalueof{/pgfplots/ymin}\\
			1	\pgfkeysvalueof{/pgfplots/ymin}\\
			1	\pgfkeysvalueof{/pgfplots/ymax}\\
			-1	\pgfkeysvalueof{/pgfplots/ymax}\\
		}--cycle;
	\end{axis}
	\begin{axis}[name=plotP,
		at={($(plotF.east)+(0.1\figurewidth,0)$)},
		anchor=west,
		scale only axis=true,
		height=0.5\figureheight,width=0.5\figurewidth,
		axis lines = left,
		axis line style = thick,
		x label style={at={(axis description cs:1,-0.01)},anchor=north},
		y label style={at={(axis description cs:-0.01,1)},anchor=south east},
		xlabel = \(x\),
		ylabel = {\(p_k(x)\)},
ymajorticks=false,
every major tick/.append style={thick},
xtick={-1,0,1},
xticklabels={{},{$v$},{}},
ymax=4, ymin=0,
xmin=-1.5, xmax=1.5,
every axis plot/.append style={thick}
		]
		\node[] at (axis cs: 0.8,3.5) {$\Cc$};
		\addplot [
		domain=-1.5:1.5, 
		samples=100, 
		color=blue,
		opacity=0.25
		]
		{1/1*x^2 };
		\addplot [
		domain=-1.5:1.5, 
		samples=100, 
		color=blue,
		opacity=0.5
		]
		{1/2*x^4 };
		\addplot [
		domain=-1.5:1.5, 
		samples=100, 
		color=blue,
		opacity=1
		]
		{1/3*x^6 };
		\addplot [
		domain=-1.2:1.2, 
		samples=100, 
		color=blue,
		opacity=1,
		dotted,
		]
		{1/15*x^30 };
		\addplot[area legend, draw=none, fill=gray, fill opacity=0.2, forget plot]
		table[row sep=crcr] {%
			x	y\\
			-1	\pgfkeysvalueof{/pgfplots/ymin}\\
			1	\pgfkeysvalueof{/pgfplots/ymin}\\
			1	\pgfkeysvalueof{/pgfplots/ymax}\\
			-1	\pgfkeysvalueof{/pgfplots/ymax}\\
		}--cycle;
	\draw[->,>=stealth,thick] (axis cs:0.65,1.2)->(axis cs:0.9,0.4);
	\end{axis}
	\begin{axis}[name=plotJ,
		at={($(plotF.south)!0.5!(plotP.south)-(0,0.15\figureheight)$)},
		anchor=north,
		scale only axis=true,
		height=1.1\figureheight,width=1.1\figurewidth,
		axis lines = left,
		axis line style = thick,
		x label style={at={(axis description cs:1,-0.01)},anchor=north},
		y label style={at={(axis description cs:-0.01,1)},anchor=south east},
		xlabel = \(x\),
		ylabel = {\(J_k(x)\)},
		ymajorticks=false,
		every major tick/.append style={thick},
		xtick={-1,0.4,0.68940,0.787618,0.954776,1},
		xticklabels={{},{$x_1^\star$},{$x_2^\star$~~},{~$x_3^\star$},{~$x_N^\star$},{}},
		ymax=4, ymin=0,
		xmin=-1.5, xmax=1.5,
		every axis plot/.append style={thick}
		]
		\node[] at (axis cs: 0.9,3.75) {$\Cc$};
		\addplot [
		domain=-1.5:1.5, 
		samples=100, 
		color=red,
		dashed,
		]
		{0.25*x^2 - 1*x + 1};
		\addplot [
		domain=-1.5:1.5, 
		samples=100, 
		color=black,
		opacity=0.25
		]
		{1/1*x^2 +0.25*x^2 - 1*x + 1}
		coordinate[pos=1] (A);
		\addplot[mark=x, opacity=0.25] coordinates
		{(0.4,0.8)}; 
		\addplot [
		domain=-1.5:1.5, 
		samples=100, 
		color=black,
		opacity=0.5
		]
		{1/2*x^4 +0.25*x^2 - 1*x + 1}
		coordinate[pos=1] (B);
		\addplot[mark=x, opacity=0.5] coordinates
		{(0.68940,0.54236)}; 
		\addplot [
		domain=-1.5:1.5, 
		samples=100, 
		color=black,
		opacity=1
		]
		{1/3*x^6 +0.25*x^2 - 1*x + 1}
		coordinate[pos=1] (C);
		\addplot[mark=x, opacity=1] coordinates
		{(0.787618,0.447042)}; 
		\addplot [
		domain=-1.2:1.2, 
		samples=100, 
		color=black,
		opacity=1,
		dotted
		]
		{1/15*x^30 +0.25*x^2 - 1*x + 1}
		coordinate[pos=0.6505] (D);
		\addplot[mark=x, opacity=1] coordinates
		{(0.954776,0.289756)}; 
		\addplot[area legend, draw=none, fill=gray, fill opacity=0.2, forget plot]
		table[row sep=crcr] {%
			x	y\\
			-1	\pgfkeysvalueof{/pgfplots/ymin}\\
			1	\pgfkeysvalueof{/pgfplots/ymin}\\
			1	\pgfkeysvalueof{/pgfplots/ymax}\\
			-1	\pgfkeysvalueof{/pgfplots/ymax}\\
		}--cycle;
	\draw[->,>=stealth,thick] (axis cs:0.4,1.2)->(axis cs:0.93,0.55) node[midway,above,sloped] {$k\to\infty$}; 
	\end{axis}
		\node[anchor=west,xshift=-0.2em] at (A) {\footnotesize$k=1$};
		\node[anchor=west,xshift=-0.2em] at (B) {\footnotesize$k=2$};
		\node[anchor=west,xshift=-0.2em] at (C) {\footnotesize$k=3$};
		\node[anchor=west,xshift=-0.2em] at (D|-C) {\footnotesize$k=N$};
\end{tikzpicture} 
	\caption{Example for the sequence of cost functions $\{J_k\}_k$ based on the original cost function $f$ and the sequence of penalty functions $\{p_k\}_k$ for the set $\Cc$ for $k\in\{1,2,3, N\}$ with $N=15$.
	The sequence of minimizers $\{x_k^\star\}_k$ converges to the minimizer $x^\star$ of the constrained problem as $k\to\infty$ (cf.\ Proposition~\ref{lem:auxMin:convegence}).}\label{fig:mainIdea}
\end{figure}

\subsection{Main idea}
Our solution approach is based on the fact that unconstrained convex polynomial optimization problems can be solved by gradient descent algorithms only involving polynomial computations.
Since the gradient of a polynomial function is a polynomial, and gradient descent steps only involve multiplication by a step size and addition to the previous value, the overall approach is polynomial and suitable for encrypted evaluation.

However, since the constrained problem cannot be exactly transformed into a polynomial unconstrained problem, we employ techniques from interior point methods and construct a converging sequence of unconstrained auxiliary problems.
For this purpose, we introduce a sequence of polynomial penalty functions $p_k$ and a sequence of auxiliary cost functions $J_k(x) = f(x) + p_k(x)$.
The construction of the penalty is based on the observation that monomials of increasing power tend towards zero for small values and grow rapidly for large values. By using penalty polynomials that tend to zero inside the allowed set and grow rapidly outside, our approach is a novel intermediate concept between the classical barrier and penalty methods.
A key distinction of our approach is that instead of defining the sequence of problems only by adjusting a weighting parameter, the shape itself of our penalty function changes.
Specifically, the penalty term is chosen such that the auxiliary cost functions have their minimum within the feasible set and this minimizer converges to a minimizer of the original problem as the index $k$ increases.
An exemplary sequence of cost functions is depicted in Figure~\ref{fig:mainIdea}.
Since the penalty function cannot approach infinity at the boundary of the feasible set, and we cannot use Newton-style iterations, we need new concepts and proof techniques.

To solve this sequence of problems and thus find a solution to the original constrained problem, we apply a sequential gradient descent algorithm to the sequence of auxiliary problems. After each gradient step, the index of the cost function increases by one. Intuitively, the gradient steps track the auxiliary minimizer, which converges to the original minimizer, resulting in a convergent algorithm.

We make the construction of the penalty function and the minimization algorithm precise in the following sections.

\section{Sequence of unconstrained problems}\label{sec:Sequence}
In this section, we introduce the sequence of auxiliary unconstrained problems and analyze its properties.

\subsection{Design of the unconstrained problems}

We choose the sequence of penalty function as
\begin{equation*}%
	p_k(x) = m \frac{1}{k} g(x)^{k}
\end{equation*}
for $k\in\N$ and a parameter $m\geq0$ that must satisfy a precise relationship described later.
Its gradient is given by
\begin{equation*}%
	\nabla p_k(x) = m g(x)^{k-1} \nabla g(x).
\end{equation*}
The auxiliary cost function is
\begin{equation*}%
	J_k(x) = f(x) + p_k(x)
\end{equation*}
and we obtain the auxiliary unconstrained problem
\begin{equation}\label{eq:AuxiliaryProblem}
	x_k^\star  =  \arg\min_{x}  \quad\hspace*{-0.5em} J_k(x).
\end{equation}
To make sure that every auxiliary Problem~\eqref{eq:AuxiliaryProblem} has a minimum inside the feasible set, the parameter $m$ has to be chosen such that $-\nabla J_k(x)$ is zero or points towards the interior of $\Cc$ for every point on the boundary $\partial\Cc$.
Hence, we make the following requirement.
	\begin{requirement}[Scaling for minimum inside $\Cc$]\label{req:min}
		The parameter $m$ satisfies $m\geq m_{\mathrm{min}} \coloneq \max(\hat{m}_{\mathrm{min}},0)$ with 
		\begin{equation*}%
			\begin{split}
				\hat{m}_{\mathrm{min}} = \max_{x} & \quad\hspace*{-0.5em} - \frac{\innerproduct{\nabla g(x)}{\nabla f(x)}}{\innerproduct{\nabla g(x)}{\nabla g(x)} }  \\
				\text{s.t.} & \quad\hspace*{-0.5em} g(x)=1. %
			\end{split}
		\end{equation*}
	\end{requirement}
This condition is equivalent to $\nabla J_k(x)$ and $\nabla g(x)$ forming an acute angle at the boundary ($g(x)=1$), that is
\begin{align}\label{eq:gJ}
	&&\innerproduct{\nabla g(x)}{\nabla J_k(x)} &\geq 0 &
	&\Leftrightarrow & - \frac{\innerproduct{\nabla g(x)}{\nabla f(x)}}{\innerproduct{\nabla g(x)}{\nabla g(x)} } &\leq  m.
\end{align}
This is clearly satisfied if $m$ is chosen as in Requirement~\ref{req:min}.
The value $m_{\mathrm{min}}$ can be interpreted as 
the largest ratio between the directional derivatives of $f(x)$ and $g(x)$ in the direction of $\nabla g(x)$
 at the boundary $\partial\Cc$.
While finding the exact $m_{\mathrm{min}}$ can be challenging, we note that any upper bound $m\geq m_{\mathrm{min}}$ is valid.
A negative $\hat{m}_{\mathrm{min}}$ means that the minimizer of the original problem is inside the constraint set. We require $m_{\mathrm{min}}\geq0$, since a negative penalization could render the problem nonconvex.

\subsection{Analysis of the unconstrained problems}

Consider the auxiliary unconstrained Problem~\eqref{eq:AuxiliaryProblem} satisfying Requirement~\ref{req:min}.
First, let us show that the minimizer is well-defined.
\begin{lemma}\label{lem:auxMin:exandunique}
For every $k\in\N$, Problem~\eqref{eq:AuxiliaryProblem} has a unique solution.
\end{lemma}
\begin{proof}
	The objective function $J_k$ is smooth, radially unbounded, and strictly convex. Hence, the minimum is attained at a finite value $x_k^\star$ and unique.
\end{proof}

The intuition of a vanishing penalty inside the feasible set as $k$ tends to infinity can be made precise as follows.
\begin{lemma}\label{lem:auxMin:uniform}
	The auxiliary cost function $J_k$ uniformly converges on the set $\Cc$ to the original cost function $f$ as $k\to\infty$, i.e., $\forall \varepsilon>0\,\exists N\colon\forall k\geq N \,\forall x\in\Cc\colon\, |J_k(x)-f(x)|<\varepsilon$.
\end{lemma}
\begin{proof}
	Pick $\epsilon>0$ and $N = 2\ceil{\frac{m}{\varepsilon}}$. Then, $|J_k(x)-f(x)|= p_k(x)= m\frac{1}{k}g(x)^k \leq m\frac{1}{k}\leq m\frac{1}{N} \leq \frac{1}{2}\varepsilon<\varepsilon$, where we used that $g(x)\leq1$ for $x\in\Cc$.
\end{proof}

With this, we arrive at our first important result concerning convergence towards the set of minimizers $\Xc^\star$.
\begin{proposition}\label{lem:auxMin:convegence}
		Let Requirement~\ref{req:min} hold.
	Then, the auxiliary minimizer $x_k^\star$ of Problem~\eqref{eq:AuxiliaryProblem}
	\begin{enumerate}
		\item is contained in the set $\Cc$ for every $k\in\N$, i.e., $x_k^\star\in \Cc \forall k\in\N$, \label{prop:contained} \vspace*{-\baselineskip}
		\item converges to $x^\star = \arg\min_{x\in\Xc^\star} g(x) \in\Xc^\star$
		as $k\to \infty$, i.e., $x_k^\star \to x^\star$ as $k\to\infty$. \label{prop:converges}
	\end{enumerate}
\end{proposition}
\begin{proof}
	Assume $x_k^\star\nin\Cc$. Then, there exists a point $x\in\partial\Cc$, and a vector $d\neq 0$ pointing out of $\Cc$, such that the directional derivative of $J_k$ at that point $x$ along the vector $d$ is negative.
	However, by Requirement~\ref{req:min}, the directional derivative of $J_k$ along any vector pointing out of $\Cc$ is nonnegative. By contradiction, $x_k^\star\in\Cc$.
	This proves Part~\ref{prop:contained}).

	The minimizer $x^\star= \arg\min_{x\in\Xc^\star} g(x)$ exists and is attained at a unique point since $g$ is strongly convex and $\Xc^\star$ is convex and compact.
	Let us define the sets $\Theta = \{\,x\mid g(x)\leq g(x^\star)\,\}\subseteq\Cc$ and $\Omega_k = \{\,x\mid f(x)\leq f(x^\star) + p_k(x^\star)\,\}$.
	Then, $x_k^\star\in \Theta\cap\Omega_k$.
	This is true since $x_k^\star$ is the minimizer of $J_k$, in particular,
	\begin{equation*}
		f(x_k^\star) + p_k(x_k^\star) \leq f(x^\star) + p_k(x^\star).
	\end{equation*}
	Since $f(x^\star)\leq f(x_k^\star)$, we have $p_k(x_k^\star) \leq p_k(x^\star)$, and hence $g(x_k^\star) \leq g(x^\star)$.
	Thus, $x_k^\star\in\Theta\subseteq\Cc$.
	Further, since $p_k(x_k^\star)\geq 0$, also $f(x_k^\star) \leq f(x^\star) + p_k(x^\star)$, and hence, $x_k^\star\in\Omega_k$.
	
	Since due to Lemma~\ref{lem:auxMin:uniform} $p_k(x^\star)$ uniformly converges to zero, the set $\Omega_k$ converges to the set $\Xc^\star$ as $k\to\infty$.
	Hence, $\Theta\cap\Omega_k \to \Theta\cap\Xc^\star$ as $k\to\infty$.
	Since $x^\star = \arg\min_{x\in\Xc^\star} g(x)$ is the only element of $\Theta\cap\Xc^\star$, $x_k^\star\to x^\star$.
	This proves Part~\ref{prop:converges}).
\end{proof}

The following property becomes important for the gradient method in the next section.
\begin{lemma}\label{lem:L-smooth}
	For every $k\in\N$, the auxiliary cost function $J_k$ is $L_k$-smooth inside $\Cc$, i.e., $\forall k\in\N\,\exists L_k\geq 0:\nabla^2 J_k(x) \preceq L_k I \,\forall x\in\Cc$,
	and a smoothness constant is given as $L_k= \bar{\sigma}(Q+m(4k-2)A)$.
\end{lemma}
\begin{proof}
	The proof is given in Appendix~\ref{proof:lem:L-smooth}.
\end{proof}

\section{Sequential gradient descent}
To find a solution to the problem presented in Section~\ref{sec:Sequence}, and thus to the original problem, we employ gradient descent since it only involves multiplication of the step size and the gradient, and addition to the previous iterate.

Consider a sequential gradient descent algorithm
\begin{align}\label{eq:GradDesc}
	x_{k+1}  =  x_k - \gamma_k \nabla J_k(x_k),
\end{align}
where we sequentially update the cost function after every gradient step.
We choose $0<\gamma_k\leq \frac{1}{L_k}$ as step size, and require that the sequence $\{\gamma_k^2\}_k$ is summable, whereas $\{\gamma_k\}_k$ is not. This is fulfilled, e.g., for $\gamma_k = \frac{1}{L_k}$.
Further, we make the following standard assumption for interior point methods.
\begin{assumption}
	The initial point is feasible, i.e., $x_1\in\Cc$.
\end{assumption}
The center of the constraint ellipsoid $x=v$ is always a feasible starting point; however, better warm starts might be possible.

For the following results, we impose a further requirement on the scaling parameter $m$.
\begin{requirement}[Scaling for invariance of $\Cc$]\label{req:inv}
	The parameter $m$ satisfies $m\geq m_{\mathrm{inv}} \coloneq \max(\hat{m}_{\mathrm{inv}},0) \geq m_{\mathrm{min}}$ with 
	\begin{equation*}
		\begin{split}
			\hat{m}_{\mathrm{inv}} = \min & \quad\hspace*{-0.5em} m\\
			\text{s.t.} 
& \quad\hspace*{-0.5em} \|\nabla f(x) + m \nabla g(x)\| \leq 2rL_1\cos(\phi(x))
~ \forall x\in \partial\Cc,
		\end{split}
	\end{equation*}
$r=\frac{\sqrt{\ubar{\sigma}(A)}}{\bar{\sigma}(A)}$, and
	$\cos(\phi(x)) = \frac{\left\langle \nabla f(x) + m \nabla g(x),\nabla g(x) \right\rangle}{\norm{\nabla f(x) + m \nabla g(x)} \norm{\nabla g(x)}}$
describing the angle between $\nabla f(x) + m \nabla g(x)$ and $\nabla g(x)$.
\end{requirement}
As for Requirement~\ref{req:min}, finding the exact $m_{\mathrm{inv}}$ can be challenging, however, any upper bound $m\geq m_{\mathrm{inv}}$ is valid.
\begin{remark}
	In the scalar case $n=1$, Requirement~\ref{req:min} and Requirement~\ref{req:inv} are equivalent.
\end{remark}
\commentout{\begin{proof}
	Without loss of generality, assume $v=0$, $\partial\Cc=\{-r,r\}$. 
	From Lemma~\ref{lem:L-smooth} it follows that $|\nabla f(r) + m \nabla g(r) - (\nabla f(-r) + m \nabla g(-r))| \leq 2rL_k \leq 2rL_1$.
	With $\nabla g(-r)<0$, $\nabla g(r)>0$, Requirement~\ref{req:min} yields $\nabla f(-r) + m \nabla g(-r) \leq 0 \land \nabla f(r) + m \nabla g(r)\geq 0$.
	Thus, $|\nabla f(r) + m \nabla g(r)| \leq 2rL_1 \land |\nabla f(-r) + m \nabla g(-r)| \leq 2rL_1$.
	This is equivalent to Requirement~\ref{req:inv} for $\cos(\phi(-r)) = \cos(\phi(r)) = 1$, which follows from Requirement~\ref{req:min}. 
	Similarly, from Requirement~\ref{req:inv}, $\cos(\phi(-r)) = \cos(\phi(r)) = 1$ follows. Therefore, $\nabla g(-r)(\nabla f(-r) + m \nabla g(-r)) \geq 0 \land \nabla g(r)(\nabla f(r) + m \nabla g(r))\geq 0$, which is the condition for Requirement~\ref{req:min}.
\end{proof}}%

Then, we can show the following important property.
\begin{proposition}%
	\label{thm:Inv}
	Let Requirement~\ref{req:inv} hold.
	Then, the set $\Cc$ is positively invariant under the gradient descent step~\eqref{eq:GradDesc} with the step size $0<\gamma_k\leq \frac{1}{L_k}$, i.e., if $x_k\in\Cc$, then also $x_{k+1}\in\Cc$.
\end{proposition}
\begin{proof}
	The proof is given in Appendix~\ref{proof:thm:Inv}.
\end{proof}
This invariance property allows to stop 
the sequential gradient descent
after any finite number of iterations $N$ without violation of the constraints.
Now, we can show our main result, convergence of 
the sequential gradient descent
to a minimizer of the original constrained problem. 
\begin{theorem}\label{thm:conv1}
	Let Requirement~\ref{req:inv} hold.
	Let $\{x_{k}\}_k$ be the sequence of iterates resulting from sequential gradient descent~\eqref{eq:GradDesc} and $x^\star\in\Xc^\star$ a minimizer of the original Problem~\eqref{eq:OriginalProblem}. Then,
	\begin{equation*}
		f(x_k)  \to f(x^\star) \text{ as } k\to\infty.
	\end{equation*}
\end{theorem}
\vspace*{\belowdisplayskip}
\begin{proof}
	For the proof, we use ideas from~\cite{boyd2003subgradient,Poljak1987}.
	By the update law~\eqref{eq:GradDesc} and convexity of $J_k$, we have
	\begin{align*}
		\|x_{k+1}-x^\star&\|^2 =  \|x_k - \gamma_k \nabla J_k(x_k)-x^\star\|^2\\
		&=  \|x_k -x^\star\|^2 - 2 \gamma_k \nabla J_k(x_k)^\top(x_k-x^\star) + \gamma_k^2\|\nabla J_k(x_k)\|^2\\
		&\leq  \|x_k -x^\star\|^2 - 2 \gamma_k (J_k(x_k) - J_k(x^\star)) + \gamma_k^2\|\nabla J_k(x_k)\|^2.
	\end{align*}
	By applying this inequality recursively, we obtain
	\begin{align*}
		0&\leq\|x_{k+1}-x^\star\|^2\\
		 &\leq  \|x_1 -x^\star\|^2 - 2 \sum_{i=1}^{k}\gamma_i (J_i(x_i) - J_i(x^\star)) + \sum_{i=1}^{k}\gamma_i^2\|\nabla J_i(x_i)\|^2,
	\end{align*}
	and hence,
	\begin{align*}
		&&2 \sum_{i=1}^{k}\gamma_i (J_i(x_i) - J_i(x^\star)) \leq & \|x_1 -x^\star\|^2 + \sum_{i=1}^{k}\gamma_i^2\|\nabla J_i(x_i)\|^2\\
		&\Leftrightarrow& 2 \sum_{i=1}^{k}\gamma_i (J_i(x_i) - f(x^\star)) \leq & \|x_1 -x^\star\|^2 + \sum_{i=1}^{k}\gamma_i^2\|\nabla J_i(x_i)\|^2\notag \\
		&& &+ 2m\sum_{i=1}^{k}\gamma_i \frac{1}{i} g(x^\star)^i.
	\end{align*}
	Finally, since by Lemma~\ref{lem:nonincreasing} (in Appendix~\ref{sec:auxLemmas}) $J_k(x_k)\leq J_i(x_i)$ for all $i\leq k$,
	\begin{multline}
		J_k(x_k) - f(x^\star) \leq \\
		\frac{\|x_1 -x^\star\|^2 + \sum_{i=1}^{k}\gamma_i^2\|\nabla J_i(x_i)\|^2 + 2m\sum_{i=1}^{k}\gamma_i \frac{1}{i} g(x^\star)^i}{2 \sum_{i=1}^{k}\gamma_i}.\label{eq:convProof}
	\end{multline}

	From Lemma~\ref{lem:summable} (in Appendix~\ref{sec:auxLemmas}), we have that $\gamma_i^2\|\nabla J_i(x_i)\|^2$ is summable.
	Further, we know that $\gamma_i\frac{1}{i}\leq\frac{1}{L_i}\frac{1}{i} =\frac{1}{ \bar{\sigma}(Q+m(4i-2)A) i} \leq\frac{1}{ m(4i-2)\bar{\sigma}(A) i} =\frac{1}{ (4i-2)i}\frac{1}{ m\bar{\sigma}(A) }$. Thus, $\gamma_i \frac{1}{i} g(x^\star)^i$ is summable.
	Since we required that the step size $\gamma_i$ is not summable, the denominator of the right-hand-side of~\eqref{eq:convProof} diverges, whereas the nominator converges to a finite value.
	From this, it follows that the right-hand-side of~\eqref{eq:convProof} converges to zero as $k\to\infty$.
	Hence, $J_k(x_k) - f(x^\star) \to0$ as $k\to\infty$.
	Since $J_k(x_k)\geq f(x_k)\geq f(x^\star)$, also $f(x_k) - f(x^\star) \to0$ as $k\to\infty$.

\end{proof}

\subsection{Encrypted implementation}
The sequential gradient descent steps of~\eqref{eq:GradDesc}
only involve polynomial operations. 
Computing the parameters needed by the algorithm is more difficult.
The required parameters are $m$ and the functions' parameters $Q$, $q$, $A$, $v$, and possibly $L_k$.
If the problem is known beforehand, the analysis for $m$ and $L_k$ can be done offline.
In a private implementation without problem knowledge, an upper bound of the parameters $m$ and $L_k$ can be used. Then, any problem that has lower true parameters can be solved by the algorithm; however, it will be conservative and possibly slower.

An additional challenge that is common to all encrypted algorithms is the lack of ability to evaluate a stopping criterion.
Typically, such stopping conditions involve a comparison of online obtained values to a threshold. Since this comparison is difficult to do for ciphertexts, the number of iterations $N$ of the optimization algorithm has to be set beforehand.

Another aspect to be considered in encrypted implementations is the multiplicative depth of the algorithm and the cryptosystem, respectively.
During every multiplication of encrypted numbers, the additive noise in the ciphertext that guarantees the security can be amplified.
Therefore, leveled homomorphic cryptosystems only allow for a limited number of multiplications. 
The number of iterations $N$ has to be chosen accordingly.
Fully homomorphic cryptosystems such as~\cite{Cheon2017}, however, support an infinite number of operations at the cost of higher computational complexity of the involved bootstrapping operation. 
For an overview of bootstrapping in different cryptosystems, see~\cite{Badawi2023}, and for an analysis of bootstrapping in a dynamic control context, see~\cite{Schlor24a}.

\section{Example: $\min(a,b)$}
An important special case of the considered problem is finding the minimum of two encrypted numbers $a$ and $b$, i.e.,
\begin{equation*}%
	 \begin{split}
	x^\star  =  \arg\min_{x} & \quad\hspace*{-0.5em} x\\
	\text{s.t.} & \quad\hspace*{-0.5em} x\in [\min(a,b),\max(a,b)].
\end{split}
\end{equation*}
This is a well-known problem for encrypted computations since comparing encrypted numbers is a difficult task~(cf.~\cite{Cheon2019,Cheon2020}).
We can recover the general problem formulation by setting
$Q = 0,~ q = 1,~ A = \frac{4}{(a-b)^2},~ v = \frac{a+b}{2}$ and choosing the functions 
$f(x) = x$ and $g(x) = \frac{4}{(a-b)^2} (x-\frac{a+b}{2})^2$.
The optimal slope ratio $m^\star = m_{\mathrm{inv}} = m_{\mathrm{min}}$ is given by 
\begin{align*}
	m^\star = - \min\left( \frac{\nabla f(a)}{\nabla g(a)} , \frac{\nabla f(b)}{\nabla g(b)} \right)  = \frac{|a-b|}{4}.
\end{align*}
Let us assume that we know an upper bound $m$ on $m^\star$ with $m=\alpha m^\star$ and $\alpha\geq1$.
Then, the auxiliary cost function is
\begin{align*}%
	J_k(x) &= x + \alpha m^\star\frac{1}{k}\left(\frac{4}{(a-b)^2} \left(x-\frac{a+b}{2}\right)^2\right)^k.
\end{align*}
The step size $\gamma_k$ can be chosen as
\begin{align*}
\gamma_k & = \frac{1}{L_k} = \frac{(a-b)^2}{4(4k-2)m} =  \frac{1}{\alpha}\frac{|a-b|}{4k-2}.
\end{align*}
Then, the gradient descent iterations for every $k\in\N$ are
\begin{align}
	x_{k+1}  %
	&=  x_{k} - \frac{(a-b)^2}{4(4k-2)m} - 2 A^{k-1} \left(x_{k}-\frac{a+b}{2}\right)^{2k-1} \label{eq:minABImpl:A}  \\
	&=  x_{k} - \frac{1}{\alpha}\frac{|a-b|}{4k-2} - \frac{1}{2k-1}\left(\frac{4}{(a-b)^2}\right)^{k-1} \left(x_{k} -\frac{a+b}{2}\right)^{2k-1}, \label{eq:minABImpl:analysis}
\end{align}
where~\eqref{eq:minABImpl:A} is in a form ready for implementation, and~\eqref{eq:minABImpl:analysis} will be used for analysis later.

For the encrypted implementation as in~\eqref{eq:minABImpl:A}, we require encrypted values of $a$, $b$ and $A$.
We note that the availability of $A$ is a strong assumption as it requires division of encrypted numbers. If these values are not available, encrypted division algorithms as in~\cite{Adamek2024} can be used once before the iteration starts.
Note that for the iteration with $k=1$, no $A$ is needed.
The constant $m$ can be chosen large enough prior to knowing the specific problem. The specific problem just should satisfy $m^\star\leq m$, then it can be solved by the algorithm.
Particularly, if we provide an algorithm with a value $m$, any problem with $|a-b|\leq 4m$ can be solved.
\begin{remark}
	This bound on compatible problems can be understood similar to an approximation interval if we approximated the minimum function by polynomials in the first place. However, here, we can provide guarantees on invariance of the solution and convergence to the true minimum.
\end{remark}
Note also that the minimizer $x_k^\star$ of the auxiliary problem would still converge to the true minimizer $x_k$ even if $0<m<m^\star$. Just the guarantees of the gradient descent do not hold any more but in many cases the iteration still converges.

\subsection{Accuracy of the auxiliary solution}
For this example, we can exactly calculate the minimizer $x_k^\star$ of the auxiliary problem depending on how conservative the choice of $m=\alpha m^\star$ is.
Due to strict convexity and radially unboundedness of $J_k$, $x_k^\star$ is a minimizer if and only if $\nabla J_k(x_k^\star)=0$.
For the analysis, we assume $a<b$ without loss of generality.
Together with the ansatz $x_k^\star = \frac{a+b}{2} -\varepsilon_k$,
 we obtain
\begin{align*}
	&& 0&= 1+ 2\alpha \frac{|a-b|}{4} \left(\frac{4}{(a-b)^{2}}\right)^k \left(-\varepsilon_k\right)^{2k-1} \\
	&\Leftrightarrow & \varepsilon_k&= \sqrt[2k-1]{\frac{1}{\alpha}} \frac{b-a}{2}.
\end{align*}
This means that the distance to the minimizer of the original problem can be expressed as
\begin{align*}
	x^\star -x_k^\star %
	 &= \frac{a-b}{2}\left(1-\sqrt[2k-1]{\frac{1}{\alpha}}\right),
\end{align*}
which for any $\alpha>0$ converges to zero as $k\to\infty$.
With this result, we can even determine the number of iterations $k$ for a desired precision $\delta<\frac{|a-b|}{2}$ as
\begin{align*}
	|x^\star -x_k^\star| &\leq \delta
	&\Leftrightarrow && k  &\geq \frac{1}{2} - \frac{\ln\left(\alpha\right)}{2 \ln(1-\frac{2}{|a-b|}\delta)}.
\end{align*}

\subsection{Accuracy of a single gradient step}\label{subsec:singleGradientStep}
For the case that the constraint parameter $A$ is not available, let us consider a gradient step for $k=1$, where only encrypted values of $a$ and $b$ are needed, since in
	\begin{align*}
		x_{2}  &= \frac{a+b}{2} - \frac{(a-b)^2}{8m}
	\end{align*}
all nonpolynomial operations are done with public numbers.
For $a<b$ this yields
\begin{align}
x_{2} %
		   &=\frac{a+b}{2} - \frac{1}{\alpha}\frac{|a-b|}{2}\label{eq:absFormulaAlpha}\\
  &= \frac{a\left(1+\frac{\alpha-1}{2}\right)+b\left(\frac{\alpha-1}{2}\right)}{\alpha},\notag
\end{align}
which is the exact minimum $a$ for $\alpha=1$. Further, $x_{2}\in[a, \frac{b+a}{2})$ for $\alpha\in[1,\infty)$.
Note that for $\alpha=1$,~\eqref{eq:absFormulaAlpha} recovers the well-known formula $\min(a,b) = \frac{a+b}{2} - \frac{|a-b|}{2}$~(cf.~\cite{Cheon2019,Cheon2020}).
However, if we replace $|a-b|$ na\"ively by the same upper bound $\frac{m}{4}=\alpha|a-b|\geq |a-b|$, we get
\begin{align*}
	x  &= \frac{a+b}{2} - \frac{m}{8}\\
	&= \frac{1+\alpha}{2}a + \frac{1-\alpha}{2}b,
\end{align*}
which is also the exact minimum $a$ for $\alpha=1$, but for $\alpha\in[1,\infty)$ takes values in $(-\infty,a]$. This might be less desirable than the invariance property of our proposed algorithm.

\section{Summary and Outlook}

In this paper, we presented a novel optimization algorithm to solve a special class of QCQP.
It is tailored to encrypted implementations as it explicitly only uses addition and multiplication, which are the natively supported operations of homomorphic cryptosystems.
With this, we demonstrated how this class of constrained optimization problems can be solved in an encrypted fashion without the need of a trusted third party, a multi-party protocol or na\"ive polynomial approximations of standard optimization algorithms.
For our proposed method, we showed several desirable properties, such as 
that the unique minimizers of the auxiliary unconstrained problems %
as well as
the gradient descent iterates always stay inside the feasible set and converge towards a minimizer of the original constrained problem.
Further, we showed how finding the minimum of two numbers can be formulated in our framework, and explicitly analyzed the relationship between accuracy, conservatism, and the number of iterations.

In future work, we plan to analyze the convergence speed for the general algorithm and further compare it with existing nonpolynomial barrier and penalty methods.
It would also be interesting to improve the handling of encrypted parameters in the constraints and to extend the idea to more general classes of optimization problems.

\appendix

\subsection{Proof of Lemma~\ref{lem:L-smooth}:}\label{proof:lem:L-smooth}
	From the definition of $J_k(x)$, we get
	\begin{align*}%
		&\nabla^2 J_k(x) = \nabla^2 f(x) + \nabla^2 p_k(x) \\
		&= Q + m( (k-1)g(x)^{k-2} \nabla g(x) \nabla g(x)^\top  + g(x)^{k-1} \nabla^2 g(x))\\
		&\preceq Q + m( (4k-4) A(x-v)(x-v)^\top A^\top  + 2A),
	\end{align*}
	where in the last inequality we used that the largest Hessian in the Loewner order is found at the boundary $\partial\Cc$, where $g(x)=1$.
	Now, we parameterize $x$ on the boundary $\partial\Cc$ as $(x-v)=\sqrt{A^{-1}}y$ with $\|y\|=1$ and $A=\sqrt{A^{-1}}\sqrt{A^{-1}}$. 
	Further, we observe that the singular values of $yy^\top$ fulfill $\sigma(yy^\top) = \{1,0,\dots,0\}$. 
	This yields
	\begin{align*}%
		\nabla^2 J_k(x) &\preceq Q + m( (4k-4) A\sqrt{A^{-1}}yy^\top\sqrt{A^{-1}}^\top A^\top  + 2A)\\
		&\preceq Q + m( (4k-4) A\sqrt{A^{-1}}\bar{\sigma}(yy^\top)I\sqrt{A^{-1}}^\top A^\top  + 2A)\\
		&= Q + m (4k-2) A \\
		&\preceq \bar{\sigma}(Q + m (4k-2) A)I.
	\end{align*}
\hfill $\blacksquare$

\subsection{Proof of Proposition~\ref{thm:Inv}:}\label{proof:thm:Inv}
	The proof works in three steps. First, we show that under the gradient step~\eqref{eq:GradDesc}, the image of a levelset of $g$, which is an ellipsoidal surface, is again an ellipsoidal surface. Second, we show that images that correspond to a lower level of $g$, are contained in ellipsoids that correspond to a level of $g(x)=1$.
	In the third step, we explicitly show that if $x_k\in\partial\Cc$, then $x_{k+1}\in\Cc$, which, according to the first part of the proof, bound all other levelsets inside the ellipsoid $C$, which is the $1$-sublevelset of $g$.
	
For the current iterate $x_k\in\Cc$ with $g(x_k)=c\in[0,1]$, we define the ellipsoidal levelset
$\partial\Cc_c= \{x \mid g(x)= c\}$ 
and the ellipsoidal levelset $\partial\Cc_c'= \{x \mid g_c'(x)= c\}$ with the same level $c$ for a quadratic function $g_c'(x) = (x-v')^\top A_c' (x-v')$ with parameters $v'$ and $A_c'$.
The center $v'$ of the ellipsoid $\partial\Cc_c'$ is given as $v'=v-\gamma_k (q +Qv)$.
The matrix $A_c'$ can be computed as $A_c'=T_c^{-1} A T_c^{-1}$ with the symmetric matrix $T_c=I-\gamma_k(Q+2mc^{k-1}A)$.
Then, 
\begin{align*}
	g_c'(x_{k+1}) 
	&= (x_k-v)^\top TT^{-1} AT^{-1} T (x_k-v)
	= g(x_k) = c.
\end{align*}
Thus, if $x_k \in\partial\Cc_c$, then $x_{k+1} \in\partial\Cc'_c$.

Now, we show that $\partial\Cc_c' = \{x \mid g'_c(x)= c\}$ is contained $\Cc_1' = \{x \mid g_1'(x)\leq c\}$.
Since $\Cc_c'$ and $\Cc_1'$ have the same center $v'$, the condition $\partial\Cc_c' \subseteq \Cc_1'$ is satisfied if and only if
\begin{align*}
&& \frac{1}{c}A_c'&\succeq A_1 &
\Leftrightarrow&& A &\succeq \sqrt{c}T_cT_1^{-1} A T_1^{-1}T_c\sqrt{c}.
\end{align*}
This holds if and only if $\bar{\sigma}(T_1^{-1}T_c\sqrt{c})\leq1$.
Since $T_1^{-1}T_c\sqrt{c}\succeq0$ and $T_1^{-1}\succ0$, this is equivalent to
\begin{align*}
&& T_1^{-1}T_c\sqrt{c} -I&\preceq 0\\
\Leftrightarrow&& (\sqrt{c}-1)(I-\gamma_k Q) + \gamma_km2A(\sqrt{c}c^{k-1}-1)&\preceq 0\\
\Leftarrow&& (\sqrt{c}-1)(I-\frac{1}{L_k} (Q-m2A) &\preceq 0\\
\Leftarrow&& (\sqrt{c}-1)\big(I-\frac{1}{\bar{\sigma}(Q+m2(2k-1)A)} (Q-m2A)\big) &\preceq 0,
\end{align*}
which is satisfied for all $k\geq 1$.

Finally, we show that if $x_k \in\partial\Cc$, then $x_{k+1} \in\Cc$, which implies that $\partial\Cc_1'\subseteq\Cc$.
Consider $x_k\in\partial\Cc$.
If $\nabla J_k(x_{k})=0$, then $x_{k+1}=x_k\in\Cc$. 
Now consider the case $\nabla J_k(x_{k})\neq0$.
The cosine of the angle between $\nabla J_k(x_k)$ and $\nabla g(x_k)$ is given as 
\begin{align*}
	\cos(\phi(x_k)) = \frac{\left\langle \nabla J_k(x_k),\nabla g(x_k) \right\rangle}{\norm{\nabla J_k(x_k)} \norm{\nabla g(x_k)}}.
\end{align*}
The radius of maximum curvature of the ellipsoid $\Cc$ is given by $r=\frac{\sqrt{\ubar{\sigma}(A)}}{\bar{\sigma}(A)}$.
For every point $x_k\in\partial\Cc$, a ball $B_r$ with radius $r$ can be placed such that $x_k\in\partial B_r$ and $B_r\subseteq \Cc$.
Consider a line from $x_k$ in the direction of $\nabla J_k(x_k)$, i.e., with angle $\phi$ from the normal $\nabla g(x_k)$ of the ellipsoid and the ball on $x_k$. Then, the length of the line inside $B_r$ is given as $2r\cos(\phi(x_k))$.
Thus, if the gradient step $\|x_{k+1}-x_k\| = \|-\gamma_k \nabla J_k(x_k)\|$ is not longer than the line length inside the ball, invariance is guaranteed.
The condition is equivalent to
\begin{align*}
& &	\|-\gamma_k \nabla J_k(x_k)\| &\leq 2r\cos(\phi(x_k))\\
&\Leftarrow & \|\nabla J_k(x_k)\|	&\leq 2rL_1\cos(\phi(x_k)),
\end{align*}
where we used that $\gamma_k\leq \frac{1}{L_k}\leq \frac{1}{L_1} \,\forall k\in\N$.
Hence, by Requirement~\ref{req:inv}, the gradient descent step~\eqref{eq:GradDesc} leads to $x_{k+1}\in\Cc$ if $x_k\in\partial\Cc$.
With this, we have shown $x_k\in\partial\Cc_c\subseteq\Cc \implies x_{k+1}\in\partial\Cc_c'\subseteq\Cc_1'\subseteq\Cc$.

\hfill $\blacksquare$

\subsection{Auxiliary lemmas}\label{sec:auxLemmas}
In this section, we provide some intermediate results that we need for the proof of Theorem~\ref{thm:conv1}.
\begin{lemma}\label{lem:nonincreasing}
	The sequence $\{J_k(x_k)\}_k$ is nonincreasing, particularly, $J_{k+1}(x_{k+1}) \leq J_{k}(x_{k+1}) \leq J_{k}(x_{k})$.
\end{lemma}
\begin{proof} %
The descent relation $J_{k}(x_{k+1}) \leq J_{k}(x_{k})$, is 
a standard property of gradient descent with the chosen sequence of step sizes.
From the construction of $J_k$, it follows that
\begin{align*}
	J_k(x_{k+1}) &= J_{k+1}(x_{k+1}) - m\left(\frac{1}{k+1}g(x_{k+1})-\frac{1}{k}\right)g(x_{k+1})^k\\
	&\geq J_{k+1}(x_{k+1}) - m\left(\frac{1}{k+1}-\frac{1}{k}\right)g(x_{k+1})^k\\
	&= J_{k+1}(x_{k+1}) + m\frac{1}{(k+1)k}g(x_{k+1})^k\\
	&\geq J_{k+1}(x_{k+1}),
\end{align*}
where we used that $0\leq g(x_{k+1})\leq1$.
\end{proof}

\begin{lemma}\label{lem:summable}
	The sequence $\{\gamma_k^2\|\nabla J_k(x_k)\|^2\}_k$ is summable, i.e., $\sum_{k=1}^{\infty}\gamma_k^2\|\nabla J_k(x_k)\|^2<\infty$. 
\end{lemma}
\begin{proof}\label{proof:lem:summable} %
From $L_k$-smoothness and the gradient step $x_{k+1} = x_k - \gamma_k \nabla J_k(x_k) \Leftrightarrow \nabla J_k(x_k) = \frac{x_k - x_{k+1}}{\gamma_k}$ it follows that
\begin{align*}
	J_k(x_{k+1}) &\leq J_k(x_k) + \nabla J_k(x_k)^\top (x_{k+1} - x_k) + \frac{L_k}{2}\|x_{k+1}-x_k\|^2\\
	&= J_k(x_k) - \frac{1}{\gamma_k}\|x_{k+1}-x_k\|^2 + \frac{L_k}{2}\|x_{k+1}-x_k\|^2\\
	&= J_k(x_k) - \delta_k\|x_{k+1}-x_k\|^2
\end{align*}
with $\delta_k=\left(\frac{1}{\gamma_k} - \frac{L_k}{2} \right)\geq \frac{L_k}{2}$. 
With $J_k(x_{k+1})\geq J_{k+1}(x_{k+1})$ we get
\begin{align*}
	\delta_k\|x_{k+1}-x_k\|^2 &\leq J_k(x_k) - J_{k+1}(x_{k+1}).
\end{align*}
If we sum from $k=1$ to $N-1$, we obtain
\begin{align*}
	\sum_{k=1}^{N-1}\delta_k\|x_{k+1}-x_k\|^2 &\leq J_1(x_1) - J_N(x_N)\\
	&\leq J_1(x_1) - f(x_N) < \infty.
\end{align*}
Thus, $\{\delta_k\|x_{k+1}-x_k\|^2\}_k$ is a summable sequence.
Since $\delta_k\geq \frac{L_k}{2}$, and $\|x_{k+1}-x_k\| = \gamma_k\|\nabla J_k(x_k)\|$, also $\{\frac{1}{2}L_k\gamma_k^2\|\nabla J_k(x_k)\|^2\}_k$ is summable, and also $\{L_k\gamma_k^2\|\nabla J_k(x_k)\|^2\}_k$ is summable.
Since $L_k\to\infty$ as $k\to\infty$, also $\{\gamma_k^2\|\nabla J_k(x_k)\|^2\}_k$ is summable.
\end{proof}

\bibliographystyle{IEEEtran} %
\bibliography{CDC2025}

\end{document}